\documentclass{amsart}

\usepackage{url,color}

\def\R{{\mathbb {R}}}
\def\N{{\mathbb {N}}}

\def\F{{\mathcal {F}}}

\def\D{{\mathcal {D}}}

\def\ve{\varepsilon}

\def\supp{\operatorname {\text{supp}}}

\newtheorem{teo}{Theorem}[section]
\newtheorem{lema}[teo]{Lemma}
\newtheorem{prop}[teo]{Proposition}
\newtheorem{corol}[teo]{Corollary}

\theoremstyle{remark}
\newtheorem{remark}[teo]{Remark}

\theoremstyle{definition}

\numberwithin{equation}{section}
\begin{document}
\parskip 3pt

\title[Fractional order Sobolev spaces in unbounded domains]{The concentration-compactness principle for fractional order Sobolev spaces in unbounded domains and applications to the generalized fractional  Brezis-Nirenberg problem.}

\author[J. Fern\'andez Bonder, Nicolas Saintier and A. Silva]{Juli\'an Fern\'andez Bonder, Nicolas Saintier and Anal\'ia Silva}

\address[Juli\'an Fern\'andez Bonder]{Departamento de Matem\'atica, FCEyN - Universidad de Buenos Aires
\hfill\break\indent Instituto de Matem\'atica Luis Santal\'o, IMAS - CONICET
\hfill\break \indent Ciudad Universitaria, Pabell\'on I (1428) Av. Cantilo s/n. \hfill\break \indent Buenos Aires, Argentina.}

\email{jfbonder@dm.uba.ar}

\urladdr{http://mate.dm.uba.ar/~jfbonder}

\address[Nicolas Saintier]{Departamento de Matem\'atica, FCEyN - Universidad de Buenos Aires
\hfill\break \indent Ciudad Universitaria, Pabell\'on I (1428) Av. Cantilo s/n. \hfill\break \indent Buenos Aires, Argentina.}

\email{nsaintie@dm.uba.ar}

\urladdr{http://mate.dm.uba.ar/~nsaintie}

\address[Anal\'ia Silva]{Departamento de Matem\'atica, FCFMyN, Universidad Nacional de San
Luis \hfill\break\indent Instituto de Matem\'atica Aplicada de San
Luis, IMASL, CONICET. \hfill\break\indent Italia avenue 1556, San Luis (5700), San Luis, Argentina.}

\email{acsilva@unsl.edu.ar}

\urladdr{https://analiasilva.weebly.com}

\subjclass[2010]{35R11, 46E25, 45G05}


\keywords{Concentration-compactness principle; unbounded domains; fractional elliptic-type problems}

\begin{abstract}
In this paper we extend the well-known concentration -- compactness principle for the Fractional Laplacian operator in unbounded domains. As an application we show sufficient conditions for the existence of solutions to some critical equations involving the fractional $p-$laplacian in the whole $\R^n$.
\end{abstract}

\maketitle
\section{Introduction.}
In recent years there has been an increasing amount of attention to problems involving nonlocal diffusion operators. These problems are so vast that it is impossible to give a comprehensive list of references.  Just to cite a few we refer to  \cite{DGLZ, Eringen, Giacomin-Lebowitz, Laskin, Metzler-Klafter, Zhou-Du} for some physical models, \cite{Akgiray-Booth, Levendorski, Schoutens} for some applications in finances, \cite{Constantin} for applications in fluid dynamics, \cite{Humphries, Massaccesi-Valdinoci, Reynolds-Rhodes} for application in ecology and \cite{Gilboa-Osher} for some applications in image processing.

The most emblematic non-local diffusion operator is probably  the so-called fractional Laplacian $(-\Delta)^s$, $0<s<1$, and its nonlinear generalization the fractional $p$-Laplacian $(-\Delta_p)^s$, $p>1$. The convenient functional framework for these operators are the fractional order Sobolev spaces. It is well-known that the usual Sobolev immersion theorem holds in the fractional setting, in particular when $sp<n$ where $n$ is the dimension of the ambient spaces, that the limiting exponent for the embedding into the Lebesgue space is $p^{*}_s = np/(n-sp)$. 
A challenging problem is then to provide sufficient conditions for the existence of a nontrivial solution to equations of the form 
\begin{equation}\label{MainEquation}
  (-\Delta_p)^s u =  h(x) |u|^{q-2}u + K(x)|u|^{p^{*}_s-2}u  
\end{equation} 
considered either in a bounded or unbounded subset of $\R^n$. 
In the case $s=1$ and $p=2$ we recover the famous Yamab\'e equation appearing in Riemannian geometry and studied by Aubin \cite{Aubin} and then by Brezis-Nirenberg \cite{BN}. 

The study of such critical equations relies on the study of the concentration phenomenon taking place when considering sequences of approximated solutions. The principle of concentration-compactness developed by Lions \cite{Lions} has proved to be a very useful tool. 
This principle was originally developed for local critical equation in bounded domains and was later extended to deal with local critical problem in unbounded domains by Chabrowski \cite{Chabrowski}. In the fractional setting such extension was recently obtained by Palatucci and Pisante \cite{PP} for $p=2$ and then by Mosconi et al for any $1<p<\tfrac{n}{s}$, see \cite{MPSY}, to deal with problems in bounded domains. 

The main contribution of this article is to obtain a concentration compactness principle in the fractional setting  suitable to deal with the possibility of loss of mass at infinity in the same spirit as Chabrowsky's paper \cite{Chabrowski} for the local case. 
We then apply this principle to obtain sufficient existence conditions for equations like \eqref{MainEquation} in all $\R^n$.

\medskip 

In order to state our results we need to recall some basic facts about fractional order Sobolev spaces. We refer to \cite{Triebel} and \cite{Hitchhicker} for more details. 

Given a function $v\in L^1_{loc}(\R^n)$, $0<s<1\le p<\infty$, we define its $(s,p)-$Gagliardo seminorm as
$$
[v]_{s,p}^p := \iint_{\R^n\times\R^n} \frac{|v(x)-v(y)|^p}{|x-y|^{n+sp}}\, dxdy.
$$
We denote $\D^{s,p}(\R^n)$ the closure of $C^{\infty}_c(\R^n)$ with respect to the Gagliardo seminorm $[v]_{s,p}$. 
Notice that this space can be also characterized as $\D^{s,p}(\R^n) = \{v\in L^{p^*_s}(\R^n)\colon [v]_{s,p}<\infty\}$.
We define the fractional $(s,p)-$gradient of a function $v\in \D^{s,p}(\R^n)$ as
\begin{equation}\label{sp.gradient}
|D^s v(x)|^p = \int_{\R^n}\frac{|v(x+h)-v(x)|^p}{|h|^{n+sp}}\,dh.
\end{equation}
Observe that this $(s,p)-$gradient is well defined a.e. in $\R^n$ and, moreover, $|D^s v|\in L^p(\R^n)$.

Throughout this paper, it will always be assumed that $sp<n$.
It is well known that for $v\in C^\infty_c(\R^n)$, $[v]_{s,p}<\infty$  the following {\em fractional order Sobolev inequality} holds (see, for instance \cite{Ponce-libro})
$$
\|v\|_{p^*_s} \le C [v]_{s,p},
$$
where $p^*_s = \frac{np}{n-sp}$ is the (critical) Sobolev exponent and, as usual, $\|v\|_q$ denotes the $L^q(\R^n)-$norm.
Obviously this inequality holds for any $v\in \D^{s,p}(\R^n)$. 
We can then consider the best constant in this inequality, namely 
\begin{equation}\label{S}
S:=\inf_{u\in \D^{s,p}(\R^N)}\frac{[v]^p_{s,p}}{\|v\|^p_{p^*_s}} 
= \inf_{u\in C^\infty_c(\R^n)} \frac{[v]^p_{s,p}}{\|v\|^p_{p^*_s}}.
\end{equation}

\medskip

The main result of this paper reads:

\begin{teo}\label{propCCP}
Let $\{u_k\}_{k\in\N}\subset \D^{s,p}(\R^n)$ be a weakly convergent sequence with weak limit $u$. 
 
Then there exist two bounded measures $\mu$ and $\nu$, an at most enumerable set of indices $I$,  and positive real numbers $\mu_i,\nu_i$, $i\in I$, such that the following convergence hold weakly in the sense of measures,
\begin{align}
& |D^s u_k|^{p}\,dx  \rightharpoonup \mu\ge |D^s u|^{p}\,dx + \sum_{i\in I} \mu_i \delta_{x_i},\label{Du}\\
& |u_k|^{p^*_s}\,dx \rightharpoonup \nu = |u|^{p^*_s}\,dx + \sum_{i\in I} \nu_i \delta_{x_i}, \label{u}\\
& S^\frac{1}{p}\nu_i^\frac{1}{p^*_s} \le \mu_i^\frac{1}{p}\qquad \text{for all }i\in I, \label{numu}
\end{align}
where $S = S(n,p,s)$ is Sobolev constant given by \eqref{S}. Moreover, if we define
\begin{align}
& \nu_\infty=\lim_{R\to\infty}\limsup_{k\to\infty}\int_{|x|>R}|u_k|^{p^*_s}\,dx,\label{nu.infinito}\\
& \mu_\infty=\lim_{R\to\infty}\limsup_{k\to\infty}\int_{|x|>R}|D^s u_k|^{p}\,dx,\label{mu.infinito} 
\end{align}
then
\begin{align}
& \limsup_{k\to\infty} \int_{\R^n} |D^s u_k|^{p}\,dx = \mu(\R^n) + \mu_\infty, \label{CCPinf1} \\
& \limsup_{k\to\infty} \int_{\R^n} |u_k|^{p^*_s}\,dx = \nu(\R^n) + \nu_\infty, \label{CCPinf2} \\
& S^\frac{1}{p}\nu_\infty^\frac{1}{p^*_s} \leq \mu_\infty^\frac{1}{p},  \label{CCPinf3}
\end{align}
\end{teo}

The proof of \eqref{Du}--\eqref{numu} can be easily deduced from the results in \cite{MPSY}. However we will include a more direct proof of this fact in order to make the paper self contained.

The main novelty here, as we mentioned above, is \eqref{CCPinf1}--\eqref{CCPinf3}. In order to achieve this, we follow the lines of the approach found in Chabrowski's paper \cite{Chabrowski}. However, some nontrivial technical difficulties appear due to the fact that in the local case, the gradient of a function with compact support also has compact support. In the nonlocal case, if $u\in C^\infty_c(\R^n)$ then $|D^s u|^p>0$ in $\R^n$. 

In order to overcome this difficulty, one needs to give an estimate of decay for the nonlocal gradient at infinity and, moreover, one also needs to prove a compact embedding result of $\D^{s,p}(\R^n)$ into $L^p$ with weights.

\medskip 

As an application of Theorem \ref{propCCP}, we obtain existence results for the critical problem
\begin{equation}\label{eqintro}
(-\Delta_p)^s u =  h(x) |u|^{q-2}u + K(x) |u|^{p^*_s - 2} u \qquad \text{ in }\R^n,
\end{equation} 
where $(-\Delta_p)^s$ is the so-called $p-$fractional Laplacian defined as
$$
(-\Delta_p)^s u(x) = p.v. \int_{\R^n} \frac{|u(x)-u(y)|^{p-2}(u(x)-u(y))}{|x-y|^{n+sp}}\, dy,
$$
where $p.v.$ stands for {\em in principal value}, and $p\le q < p^*_s$. 

In the local case $s=1$ this kind of equation have been the subject of an intense research activity since the seminal paper \cite{Aubin}. An exhaustive bibliography is almost imposible to establish. On the contrary, in the fractional setting $s<1$ much less is known though much effort have been dedicated very recently. Critical equation with the fractional Laplacian in bounded domains  have been considered in \cite{BCPS,CKL,S1,S2,SV1,SV2,Tan} when $p=2$ and in  \cite{MPSY} for a general $p$. 
Concerning critical equations in unbounded domain we are only aware of \cite{BM,BM2,ShangZhang}. These papers are  concerned with the linear case $p=2$. For general $p$ we only found \cite{Perera-Squassina-Yang}. Their results slightly overlaps with ours, however our approach allows us to treat more general problems.

Our existence results for equation \eqref{eqintro} in $\R^n$, are stated precisely in section 3 below.

\section{concentration compactness principle}

In the following we will need this two properties of the nonlocal $(s,p)-$gradient. The first one is a scaling property and the second one is a decay estimate for the nonlocal gradient of a function with compact support.

\begin{lema}\label{scaling}
Let $u\in \D^{s,p}(\R^n)$ and given $r>0$ and $x_0\in \R^n$ we define $u_{r,x_0}(x) = u(\tfrac{x-x_0}{r})$.

Then,
$$
|D^s u_{r,x_0}(x)|^p = \frac{1}{r^{sp}} |D^s u(\tfrac{x-x_0}{r})|^p.
$$
\end{lema}

\begin{proof}
The proof is an immediate consequence of the change of variables formula. In fact,
\begin{align*}
|D^s u_{r,x_0}(x)|^p &= \int_{\R^n}\frac{|u_{r,x_0}(x+h)-u_{r,x_0}(x)|^p}{|h|^{n+sp}}\,dh \\
&=  \int_{\R^n}\frac{|u(\frac{x+h-x_0}{r})-u(\frac{x-x_0}{r})|^p}{|h|^{n+sp}}\,dh \\
&= \frac{1}{r^{sp}}  \int_{\R^n}\frac{|u(\frac{x-x_0}{r} + k)-u(\frac{x-x_0}{r})|^p}{|k|^{n+sp}}\,dk\\
&=  \frac{1}{r^{sp}} |D^s u(\tfrac{x-x_0}{r})|^p.
\end{align*}
This finishes the proof.
\end{proof}

Now we show the decay lemma. Recall that if a function has compact support, then its gradient also has compact support. However, this is not the case for the nonlocal $(s,p)-$gradient. What one actually obtain is a decay of this nonlocal gradient given by the decay of the fractional kernel. 
\begin{lema}\label{decay}
Let $v\in W^{1,\infty}(\R^n)$ be such that $\supp(v)\subset B_1(0)$. Then, there exists a constant $C>0$ depending on $n, s, p$ and $\|v\|_{1,\infty}$ such that
$$
|D^s v(x)|^p\le C\min\{1, |x|^{-(n+sp)}\}.
$$
\end{lema}
\begin{proof}
Let us first obtain a global $L^\infty$ bound for $|D^s v|^p$. In fact,
$$
|D^s v(x)|^p = \left(\int_{|h|<1} + \int_{|h|\ge 1}\right) \frac{|v(x+h) - v(x)|^p}{|h|^{n+sp}}\, dh = I + II.
$$
These two integrals are bounded in the standard way:
$$
II \le C \int_{|h|\ge 1} \frac{1}{|h|^{n+sp}}\, dh = C
$$
and, using that $\|\nabla v\|_\infty <\infty$,
$$
I\le C \int_{|h|<1} \frac{1}{|h|^{n+sp-p}}\, dh = C.
$$

Now we consider the case where $|x|>2$ and obtain the desired decay. Observe first that $v(x)=0$ and so
$$
|D^s v(x)|^p = \int_{\R^n} \frac{|v(x+h)|^p}{|h|^{n+sp}}\, dh = \int_{|x+h|<1} \frac{|v(x+h)|^p}{|h|^{n+sp}}\, dh.
$$
Now, by a simple computation, it follows that $|h| \ge |x|-1 \ge \frac{|x|}{2}$ if $|x+h|<1$. Hence
$$
|D^s v(x)|^p \le \frac{C}{|x|^{n+sp}},
$$
as we wanted to show.
\end{proof}

Combining these lemmas \ref{scaling} and \ref{decay} we get the following
\begin{corol}\label{key.estimate}
Let $\phi\in W^{1,\infty}(\R^n)$ be such that $\supp(\phi) \subset B_1(0)$ and given $r>0$ and $x_0\in\R^n$ we define $\phi_{r,x_0}(x) = \phi(\tfrac{x-x_0}{r})$. Then
$$
|D^s\phi_{r,x_0}(x)|^p \le C \min\{r^{-sp}; r^n |x-x_0|^{-(n+sp)}\},
$$
where $C>0$ depends con $n,s,p$ and $\|\phi\|_{1,\infty}$.
\end{corol}

Finally, we need a compactness lemma with weights.
\begin{lema}\label{compacidad.pesos}
Let $0<s<1<p$ be such that $sp<n$ and let $p\le q<p^*_s$.

Let $w\in L^\infty(\R^n)$ be such that there exists $\alpha>0$ and $C>0$ such that
$$
0\le w(x)\le C |x|^{-\alpha}.
$$
Then, if $\alpha>sq - n\frac{q-p}{p}$, $\D^{s,p}(\R^n) \subset\subset L^q(w\, dx; \R^n)$. That is, for any bounded sequence $\{u_k\}_{k\in\N}\subset \D^{s,p}(\R^n)$, there exists a subsequence $\{u_{k_j}\}_{j\in\N}\subset \{u_k\}_{k\in\N}$ and a function $u\in \D^{s,p}(\R^n)$ such that $u_{k_j}\rightharpoonup u$ weakly in $\D^{s,p}(\R^n)$ and
\begin{equation}\label{cota.peso}
\int_{\R^n} |u_{k_j}(x) - u(x)|^q\, w(x)dx \to 0\quad \text{as } j\to\infty. 
\end{equation}
\end{lema}

\begin{remark}
Observe that in the case $p=q$ we need $\alpha>sp$. So if $\phi\in W^{1,\infty}(\R^n)$ has compact support, then $w=|D^s\phi|^p$ verifies the hypotheses of Lemma \ref{compacidad.pesos} with $q=p$.
\end{remark}

\begin{proof}
From the reflexivity of $\D^{s,p}(\R^n)$, the Rellich-Kondrashov theorem and a standard diagonal argument, it follows that there exists $u\in \D^{s,p}(\R^n)$ and a subsequence (that we still denote by $\{u_k\}_{k\in\N}$) such that
\begin{align*}
& u_k \rightharpoonup u \quad \text{weakly in } \D^{s,p}(\R^n)\\
& u_k \to u \quad \text{strongly in } L^q_{\text{loc}}(\R^n).
\end{align*}

It remains to see \eqref{cota.peso}.

Take $R>0$ to be chosen, and compute
$$
\int_{\R^n} |u_{k_j}(x) - u(x)|^q\, w(x)dx = \left(\int_{|x|<R} + \int_{|x|\ge R}\right) |u_{k_j}(x) - u(x)|^q\, w(x)dx = I + II.
$$
Let us first bound $II$. To this end, we use H\"older's inequality and obtain
\begin{align*}
II&\le C \left(\|u_k\|_{p^*_s}^q + \|u\|_{p^*_s}^q\right)\left(\int_{|x|\ge R} w^{\left(\frac{p^*_s}{q}\right)'}\, dx\right)^{\frac{1}{\left(\frac{p^*_s}{q}\right)'}}\\
&\le C\left(\int_{|x|\ge R} w^{\left(\frac{p^*_s}{q}\right)'}\, dx\right)^{\frac{1}{\left(\frac{p^*_s}{q}\right)'}},
\end{align*}
where we have used Sobolev-Poincar\'e inequality in the last step.

Finally, we use our decay assumption on $w$ and obtain $\lim_{R\to\infty} II = 0$ uniformly on $k\in\N$.

So given $\ve>0$ we chose $R>0$ such that $II<\ve$ for any $k\in\N$.

Next, in order to bound $I$, we just use the $L^\infty$ bound on $w$ and obtain
$$
I \le \|w\|_{\infty} \|u_k-u\|_{q; B_R}^q \to 0 \quad \text{as } k\to\infty.
$$
All these estimates together imply that
$$
\limsup_{k\to\infty} \int_{\R^n} |u_{k_j}(x) - u(x)|^q\, w(x)dx \le \ve,
$$
for every $\ve>0$. The proof is completed.
\end{proof}

We are now in position of proving the concentration compactness principle.

\begin{proof}[Proof of Theorem \ref{propCCP}]

The proof of \eqref{Du}--\eqref{numu}, can be found, for instance, in \cite{MPSY}. However, in order to make the paper self contained, we make a short sketch of the proof. The strategy is the same as the one in the seminal paper of P.L. Lions \cite{Lions}. 

First we consider the case where $u=0$. In this case, we first show that the measures $\mu$ and $\nu$ verify a {\em reverse H\"older inequality}. In fact, given $\phi\in C^\infty_c(\R^n)$ we will prove that
\begin{equation}\label{RH}
S^\frac{1}{p}\left(\int_{\R^n} |\phi|^{p^*_s}\, d\nu\right)^\frac{1}{p^*_s} \le \left(\int_{\R^n} |\phi|^p\, d\mu\right)^\frac{1}{p}
\end{equation}
Hence, from \eqref{RH} it follows exactly as in \cite{Lions} that there exists a countable set $I$, points $\{x_i\}_{i\in I}\subset \R^n$ and positive weights $\{\nu_i\}_{i\in I}, \{\mu_i\}_{i\in I}\subset \R$ such that
$$
\nu = \sum_{i\in I} \nu_i \delta_{x_i},\qquad \mu\ge \sum_{i\in I} \mu_i \delta_{x_i}.
$$
From this particular case $u=0$, the general case can be deduced as in \cite{Lions} by using the classical Brezis-Lieb Lemma \cite[Theorem 1]{Brezis-Lieb}.

Hence, we need to show \eqref{RH} and the relation between the weights $\nu_i$ and $\mu_i$ given by \eqref{numu}.

To prove \eqref{RH}, observe that, given $\phi\in C^\infty_c(\R^n)$, applying the Sobolev inequality we get
$$
S^\frac{1}{p} \|\phi u_k\|_{p^*_s}\le \| D^s(\phi u_k)\|_p.
$$
By the definition of $\nu$, it follows that $\|\phi u_k\|_{p^*_s}\to \left(\int_{\R^n} |\phi|^{p^*_s}\, d\nu\right)^\frac{1}{p^*_s}$ as  $k\to\infty$.

For the right-hand-side, we observe that
\begin{align*}
\| D^s(\phi u_k)\|_p \le& \left(\iint_{\R^n\times\R^n}|\phi(x)|^p\frac{|u_k(x+h)-u_k(x)|^p}{|h|^{n+sp}}\,dh\,dx\right)^{\frac{1}{p}}\\
&+\left(\iint_{\R^n\times\R^n}|u_k(x+h)|^p\frac{|\phi(x+h)-\phi(x)|^p}{|h|^{n+sp}}\,dh\,dx\right)^{\frac{1}{p}}
\end{align*}
where we have used Minkowski's inequality.

Now, observe that, by a simple change of variables,
\begin{align*}
\iint_{\R^n\times\R^n}|u_k(x+h)|^p\frac{|\phi(x+h)-\phi(x)|^p}{|h|^{n+sp}}\,dh\,dx &= \iint_{\R^n\times\R^n}|u_k(y)|^p\frac{|\phi(y)-\phi(y+\hat h)|^p}{|\hat h|^{n+sp}}\,d\hat h\,dy\\
&= \int_{\R^n} |u_k(y)|^p |D^s\phi(y)|^p\, dy.
\end{align*}

Hence, we get
$$
\|D^s(u_k\phi)\|_p \leq \left(\int_{\R^n} |\phi(x)|^p|D^s u_k(x)|^p\,dx\right)^{\frac{1}{p}}+\left(\int_{\R^n} |u_k(x)|^p |D^s \phi(x)|^{p}\,dx\right)^{\frac{1}{p}}.
$$

Now, from Lemma \ref{decay}, the weight $w(x) := |D^s\phi(x)|^p$ satisfies the hypotheses of Lemma \ref{compacidad.pesos}, and hence $u_k\to 0$ strongly in $L^p(w)$. Therefore
$$
\limsup_{k\to\infty} \|D^s(\phi u_k)\|_p\le \left(\int_{\R^n} |\phi|^p\,d\mu\right)^{\frac{1}{p}}.
$$
This concludes the proof of the reverse H\"older inequality \eqref{RH}.

Now, to prove the relation between the weights $\nu_i$ and $\mu_i$ \eqref{numu}, we take $\phi\in C^\infty_c(\R^n)$ be such that $0\le \phi\le 1$, $\phi(0)=1$, $\supp\phi = B_1(0)$ and given $\ve>0$ we consider the rescaled functions $\phi_{i,\ve}(x) = \phi(\frac{x-x_i}{\ve})$.

Without loss of generality we may assume that $x_i=0$ and write $\phi_\ve = \phi_{i,\ve}$. Recall that from Corollary \ref{key.estimate} we have that
\begin{equation}\label{estimate.ep}
|D^s \phi_\ve(x)|^p \le C \min\{\ve^{-sp}; \ve^n |x|^{-(n+sp)}\}.
\end{equation}

Now, \eqref{estimate.ep} implies that $|D^s\phi_\ve|^p$ satisfies the hypotheses of Lemma \ref{compacidad.pesos}, therefore, arguing as in the proof of the reverse H\"older inequality \eqref{RH}, one arrives at
$$
S^\frac{1}{p}\left(\int_{\R^n} |\phi_\ve|^{p^*_s}\, d\nu\right)^\frac{1}{p^*_s} \le \left(\int_{\R^n} |\phi_\ve|^p\, d\mu\right)^\frac{1}{p} + \left(\int_{\R^n} |u|^p |D^s\phi_\ve|^p\, dx\right)^\frac{1}{p}.
$$
Now, $\int_{\R^n} |\phi_\ve|^{p^*_s}\, d\nu\ge \nu_i$ and $\int_{\R^n} |\phi_\ve|^p\, d\mu\le \mu(B_\ve(0)) \to \mu_i$  as $\ve\to 0$. Hence it remains to check that 
\begin{equation}\label{limite.dificil}
\int_{\R^n} |u|^p |D^s\phi_\ve|^p\, dx\to 0\quad \text{as } \ve\to 0.
\end{equation}

Two main difficulties arise now. One, once again, comes from the fact that the nonlocal gradient $|D^s\phi_\ve|^p$ does not has compact support. The second one, unlike the bounded domain case, is that $u$ does not belong to $L^p(\R^n)$. In order to overcome these difficulties, we use the precise rate of decay for $|D^s\phi_\ve|^p$ given by \eqref{estimate.ep}.

So,
$$
\int_{\R^n}|u|^p |D^s\phi_\ve|^p\, dx \le C\left(\ve^{-sp}\int_{|x|<\ve} |u|^p\, dx + \ve^n\int_{|x|\ge \ve} \frac{|u|^p}{|x|^{n+sp}}\, dx\right) = C (I + II).
$$

The first term is the easiest one,
$$
I\le \ve^{-sp} \left(\int_{|x|<\ve} |u|^{p^*_s}\, dx\right)^{\frac{p}{p^*_s}} |B_\ve|^{\frac{sp}{n}} = C \left(\int_{|x|<\ve} |u|^{p^*_s}\, dx\right)^{\frac{p}{p^*_s}}.
$$
Since $u\in L^{p^*_s}(\R^n)$ the last term goes to zero as $\ve\to 0$.

For the second term we proceed as follows,
\begin{align*}
II &= \sum_{k=0}^\infty \ve^n \int_{2^k\ve\le |x|\le 2^{k+1}\ve} \frac{|u|^p}{|x|^{n+sp}}\, dx\\
&\le  \sum_{k=0}^\infty \frac{1}{2^{k(n+sp)}} \frac{1}{\ve^{sp}} \int_{|x|\le 2^{k+1}\ve} |u|^p\, dx\\
&\le \sum_{k=0}^\infty \frac{1}{2^{k(n+sp)}} \frac{1}{\ve^{sp}} \left(\int_{|x|<2^{k+1}\ve} |u|^{p^*_s}\, dx\right)^{\frac{p}{p^*_s}} |B_{2^{k+1}\ve}|^\frac{sp}{n}\\
&= c \sum_{k=0}^\infty \frac{1}{2^{nk}} \left(\int_{|x|<2^{k+1}\ve} |u|^{p^*_s}\, dx\right)^{\frac{p}{p^*_s}},
\end{align*}
where $c$ depends only on $n,s,p$.

Now, given $\delta>0$, take $k_0\in\N$ such that $c\sum_{k=k_0+1}^\infty 2^{-nk} <\delta$. So
\begin{align*}
II &\le \|u\|_{p^*_s}^p \delta + c \sum_{k=0}^{k_0} \frac{1}{2^{nk}}  \left(\int_{|x|<2^{k_0+1}\ve} |u|^{p^*_s}\, dx\right)^{\frac{p}{p^*_s}}\\
&= \|u\|_{p^*_s}^p \delta + C(s,p,n,k_0)\left(\int_{|x|<2^{k_0+1}\ve} |u|^{p^*_s}\, dx\right)^{\frac{p}{p^*_s}}.
\end{align*}
Therefore, we obtain that $\limsup_{\ve\to 0} II \le  \delta\|u\|_{p^*_s}^p,$ for any $\delta>0$. This concludes the proof of \eqref{numu}.

It remains to see \eqref{CCPinf1}--\eqref{CCPinf3}. Consider a smooth function $\phi\colon [0,+\infty)\to [0,1]$ such that $\phi\equiv 0$ in $[0,1]$ and $\phi\equiv 1$ in $[2,+\infty)$. Then $\phi_R(x):=\phi(|x|/R)$ is smooth and satisfies $\phi_R(x)=1$ for $|x|\ge 2R$, $\phi_R(x)=0$ for $|x|\le R$ and $0\leq\phi_R(x)\leq1$. We then write that
\begin{equation}\label{mu1}
\int_{\R^n}|D^s u_k|^{p}\,dx = \int_{\R^n}|D^s u_k|^{p}\phi_R^p\,dx + \int_{\R^n}|D^s u_k|^{p}(1-\phi_R^p)\,dx.
\end{equation}
Observe first that
$$ 
\int_{|x|>2R}|D^s u_k|^{p}\,dx \leq \int_{\R^n} |D^s u_k|^{p}\phi_R^p\, dx \leq \int_{|x|>R} |D^s u_k|^p \, dx
$$
so that
\begin{equation}\label{mu2}
\mu_\infty = \lim_{R\to\infty}\limsup_{k\to\infty} \int_{\R^n} |D^s u_k|^p \phi_R^p\, dx.
\end{equation}
In the same way
\begin{equation}\label{nu2}
\nu_\infty = \lim_{R\to\infty}\limsup_{k\to\infty} \int_{\R^n} |u_k|^{p^*_s} \phi_R^{p^*_s}\, dx.
\end{equation}
On the other hand, since $1-\phi_R^p$ is smooth with compact support,
we have by definition of $\mu$ that for $R$ fixed,
$$ 
\lim_{k\to \infty}  \int_{\R^n} (1-\phi_R^p) |D^s u_k|^p\, dx = \int_{\R^n} (1-\phi_R^p)\, d\mu.  
$$

Since $\phi_R\to 0$ pointwise and $\mu$ is a finite nonnegative measure, it follows from the Dominated Convergence Theorem that $\lim_{R\to \infty}\int_{\R^n} \phi_R^p\,d\mu = 0$. Hence, we obtain
\begin{equation}\label{mu3}
\lim_{R\to\infty}\limsup_{k\to\infty} \int_{\R^n} (1-\phi_R^p) |D^s u_k|^p\, dx = \mu(\R^n).
\end{equation}
Plugging \eqref{mu2} and \eqref{mu3} into \eqref{mu1} yields
\eqref{CCPinf1}. The proof of \eqref{CCPinf2} is similar.

By definition of the Sobolev constant $S$, we know that
$$
S^\frac{1}{p}\|u_k \phi_R\|_{p^*_s} \leq \| D^s(u_k \phi_R)\|_p.
$$ 
As before, we have
$$
\| D^s(u_k \phi_R)\|_p\leq \left(\int_{\R^n} |\phi_R(x)|^p|D^s u_k(x)|^p\,dx\right)^{\frac{1}{p}}+\left(\int_{\R^n} |u_k(x)|^p |D^s \phi_R(x)|^{p}\,dx\right)^{\frac{1}{p}}.
$$

In order to finish the proof of the result, it remains to prove that
\begin{equation}\label{to0}
\lim_{R\to\infty} \lim_{k\to\infty} \int_{\R^n}|u_k(x)|^p|D^s\phi_R(x)|^{p}\,dx = 0.
\end{equation}

Once again, we use the precise rate of decay of this nonlocal gradient given by Corollary \ref{key.estimate} and use the compact embedding of $\D^{s,p}(\R^n)$ into a weighted $L^p$ space given by Lemma \ref{compacidad.pesos}.

In fact, from Lemma \ref{compacidad.pesos} we have that
$$
\lim_{k\to\infty} \int_{\R^n} |u_k|^p |D^s\phi_R|^p\, dx = \int_{\R^n} |u|^p |D^s\phi_R|^p\, dx.
$$ 

Now, let $W(x) := \min\{1; |x|^{-(n+sp)}\}$. Hence, from Corollary \ref{key.estimate} applied to $1-\phi_R$ there exists a constant $C>0$, independent of $R$, such that $|D^s\phi_R(x)|^p \le C W(x)$ for every $R>1$. Moreover, from Lemma \ref{compacidad.pesos} we have that $|u|^p W\in L^1(\R^n)$. Finally, observe that $|D^s\phi_R(x)|^p\to 0$ as $R\to\infty$. So, by the Dominated Convergence Theorem, we arrive at
\begin{equation}\label{limite.R}
\lim_{R\to\infty} \int_{\R^n} |u|^p |D^s\phi_R|^p\, dx = 0.
\end{equation}
The proof is finished.
\end{proof}

\section{Applications to critical equations with the fractional $p$-Laplacian in $\R^n$.}

In this section we use Theorem \ref{propCCP} to obtain some existence results for the equation
\begin{equation}\label{MainEqu}
 (-\Delta_p)^s u = \lambda h(x) |u|^{q-2}u + K(x)|u|^{p^{*}_s-2}u \qquad \text{ in }\R^n,
 \end{equation} 
 where $p\le q<p^*_s$.
 
We consider two cases.
\begin{enumerate}
\item $p<q<p^*_s$ and
\item $q=p$.
\end{enumerate}

For the first case, we impose the following assumptions on $h$ and $K$:
\begin{align}
&\tag{$h_1$} \label{h1} 0\le h\in L^1_{loc}(\R^n) \text{ is such that the immersion $\D^{s,p}(\R^n)\subset L^q(h\, dx; \R^n)$ is compact.}\\
&\tag{$K_1$} \label{K1} \text{The function $K$ is nonnegative, bounded, and has a limit at $\infty$} \\
&\notag \text{i..e. } K\in L^\infty(\R^n),\ K\ge 0,\ \text{ there exists } K(\infty):=\lim_{|x|\to+\infty}K(x).
\end{align}
Under these assumptions we have the following result:

\begin{teo}\label{aplicacion.1}
Let $0<s<1<p<q$ be such that $sp<n$ and $q<p^*_s$. Assume moreover that the functions $h$ and $K$ satisfy \eqref{h1} and \eqref{K1}.

Then, there exists $\lambda_0>0$ such that \eqref{MainEqu} has a nontrivial solution for any $\lambda>\lambda_0$.
\end{teo}

This kind of result goes back to \cite{AzoreroPeral}. A result similar to ours in the linear case $p=2$ is given in \cite{ShangZhang}. 

\medskip

For the 2nd  case $q=p$ we  assume that $K(x) = 1$ and that $h$ verifies \eqref{h1} with $q=p$ and 
\begin{align}
&\tag{$h_2$}\label{h2} h\in L^\infty(\R^n) \text{ and there exists $x_0\in\R^n$ such that $h$ is continuous at $x_0$ and $h(x_0)>0$.}
\end{align}
Then, we have the following result:

\begin{teo}\label{aplicacion.2}
Let $0<s<1<p<\infty$ be such that $sp^2<n$. Assume moreover that the function $h$ satisfies \eqref{h1} with $q=p$ and \eqref{h2} and that $K(x)=1$.

Then, \eqref{MainEqu} has a nontrivial solution for any $0<\lambda<\lambda_1(h)$, where $\lambda_1(h)$ is given by
\begin{equation}\label{eq.lambda1}
\lambda_1(h):= \inf_{v\in \D^{s,p}(\R^n)} \frac{[v]_{s,p}^p}{\int_{\R^n} h(x) |v|^p\, dx}.
\end{equation}
\end{teo}

\begin{remark}\label{lambda1}
Observe that our assumption \eqref{h1} implies that $\lambda_1(h)$ is well define and positive. In fact, it is the first eigenvalue of the problem
$$
(-\Delta_p)^s u = \lambda h(x) |u|^{p-2}u \qquad \text{in } \R^n.
$$
\end{remark}

The result of Theorem \ref{aplicacion.2} can be  generalized for nonconstant  $K$ under the assumption that $K$ reaches its maximum at some $x_0\in\R^n$ and  is flat enough near $x_0$. See Theorem \ref{aplicacion.2.5}.

As a final application we study the case where $q=p$ and $K$ reaches its maximum at $\infty$. For this case we impose
\begin{align}
&\tag{$K_\infty$}\label{K.infty} 
0\leq K\in L^\infty(\R^n),\ \|K\|_\infty = \lim_{|x|\to\infty} K(x) =: K(\infty) \text{ and } |K(x)-K(\infty)|\le \frac{C}{|x|^\beta}
\end{align}
for some $\beta>0$.

Moreover, since \eqref{h1} implies a decay of $h$ at infinity, we need a hypothesis to control this decay. So we assume
\begin{align}
&\tag{$h_\infty$}\label{h.infty} h(x)\ge \frac{A}{|x|^\gamma},\quad \text{with } \gamma<\frac{n}{p-1}.
\end{align}
for large values of $|x|$.

\begin{teo}\label{aplicacion.infty}
Let $0<s<1<p<\infty$ be such that $sp^2<n$. Assume moreover that the function $h$ satisfies \eqref{h1} with $q=p$ and \eqref{h.infty} and that $K$ verifies \eqref{K.infty} for some $\beta>\gamma$.

Then, \eqref{MainEqu} has a nontrivial solution for any $0<\lambda<\lambda_1(h)$, where $\lambda_1(h)$ is given by \eqref{eq.lambda1}.
\end{teo}

\begin{remark}
Observe that from Lemma \ref{compacidad.pesos}, in order for $h$ to satisfy \eqref{h1}  and \eqref{h.infty} it is enough to have 
$$
\frac{A}{|x|^\gamma}\le h(x)\le \frac{C}{|x|^{\gamma'}},
$$
for large values of $|x|$, with $sp<\gamma'\le \gamma<\frac{n}{p-1}$. Since we are assuming $sp^2<n$ we have that $sp<\frac{n}{p-1}$ so there always exists admissible values for $\gamma$ and $\gamma'$.
\end{remark}

\medskip 

The method of proof of Theorems \ref{aplicacion.1}, \ref{aplicacion.2},  \ref{aplicacion.2.5} and 
\ref{aplicacion.infty} we just stated is quite standard and relies on the standard Mountain Pass Theorem. 
The Concentration-Compactness Principle Theorem \ref{propCCP} is then used to obtain an existence criterion (see Theorem \ref{teoMP} below). 
We then conclude the proofs of the different results doing some test-functions computations from which we obtain the sufficient conditions stated above.

This scheme of proof is very standard when dealing with critical equations, local or non-local. 
In the nonlocal setting it has been succesfuly used to deal with the fractional Laplacian in bounded or unbounded domain in all the papers mentioned in the introduction. It must be mentioned however that the authors in \cite{Tan} and \cite{BM} adopt Cafarrelli-Silvestre's approach of the fractional Laplacian considering the extension of the equation to $\R^{n+1}$ which allows to recover a local setting at the cost of dealing with degenerate operators, see \cite{CS}.
 
Another approach has been taken in \cite{BM2} where the authors construct a solution using the Lyapunov-Schmidt reduction method considering equation \eqref{MainEqu} for small values of $\lambda$ and thus as a perturbation of the pure critical case. 

The existence conditions we obtained are the natural counterpart of the known results in the local case $s=1$ and agrees with the results obtained in the fractional setting in \cite{MPSY} and \cite{SV3}. It must be mentioned however that our last result, though using the same ideas,
does not seem to be classical even in the local case (we refer to \cite{SaintierSilva} where this kind of result has been proved recently in the context of variable exponent spaces).

\subsection{An existence criterion}

We will look for a solution of \eqref{MainEqu} as a critical point of the associated functional
$$
\F_\lambda\colon \D^{s,p}(\R^n)\to\R\\
$$
\begin{equation}\label{DefF}
\F_\lambda(u) := \frac{1}{p}\iint_{\R^n\times \R^n}\frac{|u(x)-u(y)|^p}{|x-y|^{n+sp}}\,dxdy - \frac{\lambda}{q}\int_{\R^n} h(x) |u|^q\, dx - \frac{1}{p^*_s} \int_{\R^n} K(x)|u|^{p^*_s}\, dx.
\end{equation}
We first prove a  preliminary lemma which is more or less classical.

\begin{lema}\label{acotada1}
Let $\{u_k\}_{k\in\N} \subset \D^{s,p}(\R^n)$ be a Palais-Smale sequence for $\F_\lambda$. Then, up to a subsequence, there exists $u\in \D^{s,p}(\R^n)$ such that $u_k\rightharpoonup u$ weakly in $\D^{s,p}(\R^n)$ and $u$ is a weak solution of \eqref{MainEqu}.

Moreover letting $\mu$, $\nu$, $\mu_i$, $\nu_i$, $\mu_\infty$, $\nu_\infty$ be as in Theorem \ref{propCCP} when applied to $\{u_k\}_{k\in\N}$ we have the following estimates:
\begin{align}
&  \nu_i \ge  S^\frac{n}{sp} K(x_i)^{-\frac{n}{sp}}, \quad  \mu_i \ge   S^\frac{n}{sp} K(x_i)^{1-\frac{n}{sp}} \qquad \text{if } K(x_i)>0,  \label{CotaInf1} \\
& \mu_i = \nu_i=0 \qquad \text{ if } K(x_i)= 0, \label{CotaInf2}
\end{align}
and a similar result at infinity:
\begin{align}
& \nu_\infty \ge  S^\frac{n}{sp} K(\infty)^{-\frac{n}{sp}}, \quad \mu_\infty \ge  S^\frac{n}{sp}K(\infty)^{1-\frac{n}{sp}} \qquad \text{if } K(\infty)>0, \label{CotaInf3}  \\
& \mu_\infty = \nu_\infty=0 \qquad \text{ if } K(\infty)= 0, \label{CotaInf4}
\end{align}
where $K(\infty) = \limsup_{|x|\to\infty} K(x)$.
\end{lema}
\begin{proof}
The proof is more or less classical so we will be sketchy. We first prove that $\{u_k\}_{k\in\N}$ is bounded in $\D^{s,p}(\R^n)$. Recalling the definition of a Palais-Smale sequence, it is easily seen that
\begin{align*}
C+o(1) [u_k]_{s,p} &\geq \F_\lambda(u_k) - \frac{1}{q} \langle \F_\lambda'(u_k), u_k\rangle\\
&\geq \left(\frac{1}{p} - \frac{1}{q}\right) [u_k]_{s,p}^p + \left(\frac{1}{q} - \frac{1}{p^*_s}\right) \int_{\R^n}  K(x) |u_k|^{p^*_s}\, dx.
\end{align*}
Therefore, from \eqref{K1} we conclude that $\{u_k\}_{k\in\N}$ is bounded in $\D^{s,p}(\R^n)$.

Up to a subsequence we can thus assume that $\{u_k\}_{k\in\N}$ weakly converges in $\D^{s,p}(\R^n)$ to some $u$, and then also that the convergence holds in $L^q(h\, dx, \R^n)$.

Estimates \eqref{CotaInf1}--\eqref{CotaInf4} are a direct consequence
of \eqref{numu}, \eqref{CCPinf3} and the following:
\begin{equation}\label{Cota}
\mu_i = \nu_i K(x_i) \qquad \text{for any }i\in I,
\end{equation}
and  
\begin{equation}\label{CotaInfinity}
\mu_\infty = \nu_\infty K(\infty).
\end{equation}
To prove \eqref{Cota}, we fix a concentration point $x_i$, a smooth function $\phi\colon \R^n\to [0,1]$  with compact support in $B_2$  such that $\phi=1$ in $B_1$, and consider $\phi_\delta(x):=\phi(\frac{|x-x_i|}{\delta})$. 

Notice that the sequence $\{u_k \phi_\delta\}_{k\in\N}$ is bounded in $\D^{s,p}(\R^n)$. We then write that
\begin{align*}
\langle(-\Delta_p)^s u_k, u_k\phi_\delta\rangle =& \iint_{\R^n\times\R^n} \frac{|u_k(x)-u_k(y)|^p}{|x-y|^{n+sp}} \phi_\delta(x)\, dxdy\\
&+ \iint_{\R^n\times \R^n} \frac{|u_k(x)-u_k(y)|^{p-2}(u_k(x)-u_k(y))(\phi_\delta(x)-\phi_\delta(y))}{|x-y|^{n+sp}}u_k(y)\, dxdy\\
=&  I + II\\
\end{align*}
For the first term, we have
$$
I = \int_{\R^n}\phi_\delta |D^s u_k|^p\, dx\to \int_{\R^n} \phi_\delta\, d\mu.
$$
But, since $\mu - \sum_{i\in I}\mu_i \delta_{x_i}$ has no atoms and $\phi_\delta (x)\to 0$ as $\delta\to 0$ for any $x\neq x_1$, we conclude that
$$
\lim_{\delta\to 0} \lim_{k\to\infty} I = \mu_i.
$$
The second term converges to 0. In fact,
\begin{align*}
II &\leq \int_{\R^n} |u_k(y) |\left(|D^s u_k(y)|^p\right)^{\frac{1}{p'}} \left(|D^s\phi_\delta(y)|^p\right)^{\frac{1}{p}}\,dy\\
& \le \|D^s u_k\|_p^\frac{p}{p'} \left(\int_{\R^n} |u_k|^p |D^s\phi_\delta|^p\, dx\right)^\frac{1}{p}\\
&\le C \left(\int_{\R^n} |u_k|^p |D^s\phi_\delta|^p\, dx\right)^\frac{1}{p}.
\end{align*}
Using now Lemmas \ref{decay} and \ref{compacidad.pesos}, we get that
$$
\limsup_{k\to\infty} II \le C \left(\int_{\R^n} |u|^p |D^s\phi_\delta|^p\, dx\right)^\frac{1}{p}.
$$
Finally, arguing as in the proof of \eqref{limite.dificil}, it follows that 
$$
\lim_{\delta \to 0}\lim_{k\to\infty} II =  0.
$$

On the other hand as $\{u_k\}_{k\in\N}$ is  a Palais-Smale sequence,
$$
o(1)=\langle\F_\lambda'(u_k), u_k\phi_\delta\rangle = \langle -\Delta_p^s u_k, u_k\phi_\delta\rangle - \lambda \int_{\R^n} h(x) |u_k|^q\phi_\delta\, dx - \int_{\R^n} K(x) |u_k|^{p^*_s} \phi_\delta\, dx.
$$
It is easy to check that
$$
\lim_{\delta\to0}\lim_{k\to\infty} \int_{\R^n} h(x) |u_k|^q\phi_\delta\, dx = \lim_{\delta\to0} \int_{\R^n} h(x) |u|^q \phi_\delta\, dx = 0
$$
and
$$
\lim_{\delta\to0}\lim_{k\to\infty} \int_{\R^n} K(x) |u_k|^{p^*_s}\phi_\delta\, dx = \lim_{\delta\to0} \int_{\R^n} K(x) \phi_\delta\, d\nu = K(x_i)\nu_i.
$$
We conclude that $K(x_i)\nu_i = \mu_i$. 

The proof of \eqref{CotaInfinity} is similar to the one of \eqref{Cota}. In this case, we fix $\phi\colon \R^n\to [0,1]$ be a smooth function such that $\phi\equiv 0$ in $B_1$ and $\phi\equiv 1$ in $\R^n\backslash B_2$, and then consider $\phi_R(x):=\phi(\frac{x}{R})$. Notice that for a given $R>0$, the sequence $\{u_k \phi_R\}_{k\in\N}$ is bounded in $\D^{s,p}(\R^n)$.

Now
\begin{align*}
\langle (-\Delta_p)^s u_k, u_k\phi_R\rangle =& \iint_{\R^n\times\R^n} \frac{|u_k(x)-u_k(y)|^{p}}{|x-y|^{n+sp}}\phi_R(x)\, dxdy \\
& + \iint_{\R^n\times \R^n} \frac{|u_k(x)-u_k(y)|^{p-2}(u_k(x)-u_k(y))(\phi_R(x)-\phi_R(y))}{|x-y|^{n+sp}}u_k(y)\, dxdy\\
=& I+II.
\end{align*}
For $I$ we have $\lim_{R\to\infty}\lim_{k\to\infty} I =  \lim_{R\to\infty}\lim_{k\to\infty}\int_{\R^n}\phi_R |D^s u_k|^{p}\,dx = \mu_\infty.$

Proceeding as in the proof of \eqref{Cota}, we get that
$$
\limsup_{k\to\infty} II \le C\left(\int_{\R^n}|u|^p|D^s\phi_R|^p\,dy\right)^{\frac{1}{p}}.
$$
Finally, we argue as in the proof of \eqref{limite.R} and obtain that
$$
\lim_{R\to\infty}\lim_{k\to\infty} II = 0.
$$

As before,
$$
o(1)=\langle \F_\lambda'(u_k), u_k\phi_R\rangle = \langle -\Delta_p^s u_k, u_k\phi_R\rangle - \lambda \int_{\R^n} h(x) |u_k|^q \phi_R\, dx - \int_{\R^n} K(x) |u_k|^{p^*_s}\phi_R\, dx.
$$
It is easy to see that
\begin{align*}
& \lim_{R\to\infty}\lim_{k\to\infty} \int_{\R^n} h(x) |u_k|^q \phi_R\, dx = 0,\\
& \lim_{R\to\infty}\lim_{k\to\infty} \int_{\R^n} K(x) |u_k|^{p^*_s}\phi_R\,dx = K(\infty)\nu_\infty.
\end{align*}
In fact,
\begin{align*} 
\int_{\R^n} K(x) |u_k|^{p^*_s}\phi_R\,dx 
= K(\infty)\int_{\R^n} |u_k|^{p^*_s}\phi_R\,dx
+ \int_{\R^n} (K(x)-K(\infty)) |u_k|^{p^*_s}\phi_R\,dx
= K(\infty)I + II, 
\end{align*}
and passing to the limit as $k\to\infty$ and then as $R\to\infty$ we get that $I\to \nu_\infty$ and $II\to 0$ since, given $\ve>0$, for $R$ large, we have $II\le \ve \int_{|x|\ge R}|u_k|^{p^*}dx \to \ve\nu_\infty$. 

Thus combining these estimates, we obtain \eqref{CotaInfinity}.

It follows from \eqref{CotaInf1}--\eqref{CotaInf2} that there are at most a finite number of concentration points. 

It remains to see that $u$ is a weak solution of \eqref{MainEqu}. But this is somewhat standard, since given $v\in C^\infty_c(\R^n)$, we have that
\begin{align*}
o(1) = &\langle \F_\lambda'(u_k), v \rangle \\
= &\iint_{\R^n\times\R^n} \frac{|u_k(x)-u_k(y)|^{p-2}(u_k(x)-u_k(y))}{|x-y|^\frac{n+sp}{p'}} \frac{v(x)-v(y)}{|x-y|^{\frac{n}{p}+s}}\, dxdy\\
& - \lambda \int_{\R^n} h(x) |u_k|^{q-2}u_k v\, dx - \int_{\R^n} K(x) |u_k|^{p^*_s-2}u_k v\, dx.
\end{align*}
By standard integration theory, one gets that
\begin{align*}
&\int_{\R^n} h(x) |u_k|^{q-2}u_k v\, dx  \to \int_{\R^n} h(x) |u|^{q-2}u v\, dx, \\
&\int_{\R^n} K(x) |u_k|^{p^*_s-2}u_k v\, dx \to \int_{\R^n} K(x) |u|^{p^*_s-2}u v\, dx.
\end{align*}

Now, if we call
$$
\xi_k(x,y) = \frac{|u_k(x)-u_k(y)|^{p-2}(u_k(x)-u_k(y))}{|x-y|^\frac{n+sp}{p'}},
$$
then $\{\xi_k\}_{k\in\N}$ is bounded in $L^{p'}(\R^n\times\R^n)$ and so there exists $\xi\in L^{p'}(\R^n\times\R^n)$ such that $\xi_k\rightharpoonup \xi$ weakly in $L^{p'}(\R^n\times\R^n)$.

Therefore
$$
\iint_{\R^n\times\R^n} \xi_k(x,y) \frac{v(x)-v(y)}{|x-y|^{\frac{n}{p}+s}}\, dxdy\to \iint_{\R^n\times\R^n} \xi(x,y)\frac{v(x)-v(y)}{|x-y|^{\frac{n}{p}+s}}\, dxdy.
$$
Finally, since $u_k\to u$ a.e. in $\R^n$, one obtains that $\xi_k\to \frac{|u(x)-u(y)|^{p-2}(u(x)-u(y))}{|x-y|^\frac{n+sp}{p'}}$ a.e. in $\R^n\times\R^n$ and so
$$
\xi(x,y) =  \frac{|u(x)-u(y)|^{p-2}(u(x)-u(y))}{|x-y|^\frac{n+sp}{p'}}.
$$
These facts altogether give that $u$ is a weak solution to \eqref{MainEqu}.
\end{proof}

We now prove that the functional $\F_\lambda$ verifies the Palais-Smale condition for small energy levels.

\begin{prop}\label{PropPSCond}
The functional $\F_\lambda$ defined in \eqref{DefF} verifies the Palais-Smale condition at level $c$ for all real numbers $c$ satisfying
$$ 
c< \frac{s}{N} S^{\frac{N}{sp}} \|K\|_\infty^{1-\frac{N}{sp}}.
$$
\end{prop}

\begin{proof}
Let $\{u_k\}_{k\in\N}\subset \D^{s,p}(\R^n)$ be a Palais-Smale sequence for $\F_\lambda$ of level $c$. Up to a subsequence, we can assume that $\{u_k\}_{k\in\N}$ weakly converges  to some $u\in \D^{s,p}(\R^n)$. 

Let $\mu$, $\nu$, $\mu_i$, $\nu_i$, $\mu_\infty$, $\nu_\infty$ be as in the concentration-compactness principle  Theorem. \ref{propCCP} when applied to $\{u_k\}_{k\in\N}$. Then
\begin{align*}
c &= \lim_{k\to\infty} \F_\lambda(u_k)\\
&=  \frac{1}{p} (\mu(\R^n) + \mu_\infty) - \frac{\lambda}{q}\int_{\R^n} h(x) |u|^q\,dx - \frac{1}{p^*_s} \int_{\R^n} K(x)\,d\nu \\
&\ge  \F_\lambda(u) + \sum_{i\in I} \left(\frac{\mu_i}{p} - \frac{K(x_i)\nu_i}{p^*_s}\right) + \frac{\mu_\infty}{p} - \frac{K(\infty)\nu_\infty}{p^*_s}.
\end{align*}
Observe now that, since by Lemma \ref{acotada1} $u$ is a solution to \eqref{MainEqu}, we have that
$$
\F_\lambda(u) = \lambda \left(\frac{1}{p}-\frac{1}{q}\right)\int_{\R^n} h(x) |u|^q\, dx + \left(\frac{1}{p}-\frac{1}{p^*_s}\right)\int_{\R^n} K(x) |u|^{p^*_s}\, dx\ge 0,
$$
then we deduce, using \eqref{Cota}, \eqref{CotaInfinity} and \eqref{CotaInf1}-\eqref{CotaInf4}, that
\begin{equation*}
c \ge  \frac{s}{N} S^\frac{n}{sp} \|K\|_{\infty}^{1-\frac{N}{sp}} (\# I + 1)
\end{equation*}
therefore $I=\emptyset$, $\nu_\infty=\mu_\infty=0$ and the result follows.
\end{proof}

A direct application of the Mountain-pass theorem combined with Proposition \ref{PropPSCond} then yields the following existence condition:

\begin{teo}\label{teoMP} Let $0<s<1<p\le q$ be such that $sp<n$ and $p\le q<p^*_s$. Assume \eqref{h1} and \eqref{K1}. Moreover, assume that there exists $v\in \D^{s,p}(\R^n)$ such that
\begin{equation}\label{CCPCond}
 \sup_{t>0} \F_\lambda(tv) < \frac{s}{n} S^\frac{n}{sp} \|K\|_\infty^{1-\frac{n}{sp}}.
\end{equation}
Then there exists a non-trivial  solution for \eqref{MainEqu}.
\end{teo}

\begin{proof}
We only need to check the geometric conditions of the Mountain Pass Theorem. The case $p<q<p^*_s$ is standard and is omitted.

For the case $p=q$ we just observe that if $0<\lambda<\lambda_1(h)$, then
$$
[v]_{s,p}^p - \lambda\int_{\R^n} h(x) |v|^p\, dx\ge \left(1-\frac{\lambda}{\lambda_1(h)}\right)[v]_{s,p}^p.
$$
From this inequality the rest of the proof is standard.
\end{proof}

\subsection{Case $p<q<p^*_s$. Proof of Theorem \ref{aplicacion.1}}

We only need to show the existence of $v\in \D^{s,p}(\R^n)$ such that \eqref{CCPCond} holds. To this end, fix $v\in \D^{s,p}(\R^n)$ such that 
$$
\int_{\R^n} K(x) |v|^{p^*_s}\, dx = 1,
$$
and define 
$$
\varphi_\lambda(t) = \F_\lambda(tv) = \frac{t^p}{p} [v]_{s,p}^p - \lambda \frac{t^q}{q} \int_{\R^n} h(x) |v|^q\, dx - \frac{t^{p^*_s}}{p^*_s}.
$$
It is easy to see that given $\lambda>0$, there exists $t_\lambda>0$ such that $\sup_{t>0} \F_\lambda(tv) = \varphi_\lambda(t_\lambda)$.

We will show that $t_\lambda\to 0$ as $\lambda\to\infty$ and so $\limsup_{\lambda\to\infty} \varphi_\lambda(t_\lambda)\le 0$, therefore the conclusion of Theorem \ref{aplicacion.1} follows.

Now, just observe that
$$
0 = \varphi_\lambda'(t_\lambda) = t_\lambda^{p-1} [v]_{s,p}^p - \lambda t_\lambda^{q-1} \int_{\R^n} h(x) |v|^q\, dx - t_\lambda^{p^*_s - 1},
$$
from where it follows that
\begin{equation}\label{cota.tlambda}
[v]_{s,p}^p = t_\lambda^{p^*_s - p} + t_\lambda^{q-p}\lambda\int_{\R^n}h(x) |v|^q\, dx.
\end{equation}
From \eqref{cota.tlambda} if follows that $t_\lambda$ is bounded and, moreover,
$$
t_\lambda\le \left(\frac{[v]_{s,p}^p}{\lambda\int_{\R^n} h(x) |v|^q\, dx}\right)^\frac{1}{q-p}.
$$
So, $t_\lambda\to 0$ as we wanted to show and this concludes the proof of Theorem \ref{aplicacion.1}.\qed

\begin{remark}
A careful observation of the proof of Theorem \ref{aplicacion.1} provides with a (somewhat) explicit lower bound for $\lambda_0$. In fact, is we denote
$$
C=C(n,s,p,q,h) = \inf_{v\in \D^{s,p}(\R^n)} \frac{[v]_{s,p}^p}{\left(\int_{\R^n} h(x) |v|^q\, dx\right)^\frac{p}{q}},
$$
which is positive and well defined by \eqref{h1}, then one gets
$$
\lambda_0 > C^\frac{q}{p} \left(\frac{s}{n} S^\frac{n}{sp} \|K\|_\infty^{1-\frac{n}{sp}}\right)^\frac{p-q}{p}.
$$
\end{remark}

\begin{remark}
We want to point out that, once the functional setup for the functional $\F_\lambda$ is stablished, together with the compact immersion $\D^{s,p}(\R^n)\subset \subset L^p(h, dx)$, one can obtain the existence result for \eqref{MainEqu} under the same assumptions of Theorem \ref{aplicacion.1} but for any $\lambda\ge \lambda_1(h)$. In fact, following the ideas of \cite{Perera-Squassina-Yang}, one applies a Linking theorem due to Yang and Perera in \cite{Yang-Perera} (see also \cite{Candito}) exactly in the same way as in \cite{Perera-Squassina-Yang} with the obvious modifications.
\end{remark}

\subsection{Case $p=q$. Proof of Theorem \ref{aplicacion.2}}

Again, by Theorem \ref{teoMP} it remains to show the existence of a test function $v\in \D^{s,p}(\R^n)$ such that \eqref{CCPCond} holds. Recall that in this case we are assuming that $K(x)=1$.

The idea behind this construction goes back to the seminal paper by Brezis and Nirenberg \cite{BN}. The main difficulty here is the fact that the explicit form of the extremal for the Sobolev embedding $\D^{s,p}(\R^n)\subset L^{p^*_s}(\R^n)$ is not know.

It is conjectured that this extremals are  of the form
$$
V_{\ve, x_0}(x) := \ve^{-\frac{n-sp}{p}} V\left(\frac{x-x_0}{\ve}\right),\qquad V(x) = \left(\frac{1}{1+|x|^{p'}}\right)^\frac{n-sp}{p},
$$
which are the natural extensions of the standard bubbles for the embedding $\D^{1,p}(\R^n)\subset L^{p^*}(\R^n)$.

This conjecture is only known to be true in the case $p=2$. See \cite{Lieb}.

In the general case, it remains open, but nonetheless, what is known (see \cite{BMS} for a thorough study of this problem) is that there exists an extremal $U\in\D^{s,p}(\R^n)$ for \eqref{S} and that this extremal is radial and behaves like $V$ at infinity. More precisely, it is shown in \cite{BMS} that there exists universal constants $c_1, c_2>0$ such that
\begin{equation}\label{c12}
c_1 V(x)\le U(x)\le c_2 V(x),\qquad \text{for } |x|\ge 1.
\end{equation}

Once these observations are made, the proof of Theorem \ref{aplicacion.2} is even simpler than the classical result of Brezis and Nirenberg, since we are working in the whole space and there is no need to truncate the extremal $U$. However, this approach gives a further restriction on the exponents, that is classical in the literature, that is $sp^2<n$. This restriction guarantees that $U\in L^p(\R^n)$.

So, we will show that $U_\ve$ verifies \eqref{CCPCond}, i.e.
$$
\sup_{t>0} \F_\lambda(t U_\ve) < \frac{s}{n} S^\frac{n}{sp},
$$
if $\ve>0$ is small enough, where 
$$
U_\ve(x) = \ve^{-\frac{n-sp}{p}} U\left(\frac{x-x_0}{\ve}\right),
$$
and $x_0\in \R^n$ is the point given in \eqref{h2}. Without loss of generality, we assume that $x_0=0$.

We can also assume that $U$ is normalized as
$$
S^\frac{n}{sp} = [U]_{s,p}^p = \|U\|_{p^*_s}^{p^*_s}
$$
and so
\begin{equation}\label{Uve}
S^\frac{n}{sp} = [U_\ve]_{s,p}^p = \|U_\ve\|_{p^*_s}^{p^*_s}.
\end{equation}
Therefore, using \eqref{Uve},
$$
\F_\lambda(t U_\ve) = \left(\frac{t^p}{p} - \frac{t^{p^*_s}}{p^*_s}\right) S^\frac{n}{sp} - \lambda\frac{t^p}{p}\int_{\R^n} h(x) U_\ve^p\, dx.
$$
But
$$
\int_{\R^n} h(x) U_\ve^p\, dx = \ve^{sp} \int_{\R^n} h(\ve x) U^p(x)\, dx = \ve^{sp}\left( h(0)\|U\|_p^p + \int_{\R^n} (h(\ve x) - h(0)) U^p(x)\, dx\right).
$$
So, if $sp^2<n$, \eqref{c12} implies that $U\in L^p(\R^n)$, hence, by \eqref{h2},
\begin{equation}\label{asintotico}
\int_{\R^n} h(x) U_\ve^p\, dx = \ve^{sp} h(0) \|U\|_p^p + o(\ve^{sp}).
\end{equation}

Therefore, straightforward computations show that
\begin{align*}
\sup_{t>0} \F_\lambda(t U_\ve) &= \frac{s}{n}\left(S - \lambda h(0)\|U\|_p^p S^{1-\frac{n}{sp}}\ve^{sp} + o(\ve^{sp})\right)^\frac{n}{sp} \\
&< \frac{s}{n} S^\frac{n}{sp},
\end{align*}
if $\ve>0$ is small enough. The proof is complete. \qed

\subsection{A refinement of Theorem \ref{aplicacion.2}}

The result of Theorem \ref{aplicacion.2} can be further generalized of we assume that the function $K$ in \eqref{MainEqu} is non constant, but reaches its maximum at some point $x_0$ and is flat enough near $x_0$.

In fact, we need to impose that there exists $\alpha^* = \alpha^*(n,p,s)$ such that
\begin{equation}\label{alpha}
|K(x)-K(x_0)|\le C |x-x_0|^\alpha,\quad \text{for some } \alpha>\alpha^*.
\end{equation}
In fact, from our computations it follows that we can take 
$$
\alpha^* = \frac{spn}{n-sp(p-1)}.
$$
Recall that since we are assuming that $sp^2<n$ if follows that $\alpha^*>0$.

To see this fact, is enough to prove that
\begin{equation}\label{Est.K}
\int_{\R^n} K(x) U_{\ve, x_0}(x)^{p^*_s}\, dx = K(x_0) S^\frac{n}{sp} + o(\ve^{sp}).
\end{equation}

This is the content of the next lemma.
\begin{lema}\label{lema.K}
Assume $sp^2<n$ and $K\in L^\infty(\R^n)\cap C(\R^n)$ be such there exists $x_0\in\R^n$ such that \eqref{alpha} holds.  Then \eqref{Est.K} holds true.
\end{lema}

\begin{proof}
The proof is rather standard. Without loss of generality we can assume that $x_0=0$.
\begin{align*}
\int_{\R^n} K(x) U_\ve(x)^{p^*_s}\, dx &= \int_{\R^n} K(\ve x) U(x)^{p^*_s}\, dx\\
&= K(0) S^\frac{n}{sp} + \int_{\R^n} (K(\ve x) - K(0)) U(x)^{p^*_s}\, dx.
\end{align*}
So it suffices to show that
$$
\int_{\R^n} |K(\ve x) - K(0)| U(x)^{p^*_s}\, dx = o(\ve^{sp}).
$$
As usual, we split the integral for small and large values of $|x|$. For that purpose, we take $R>0$ to be defined later and write
$$
\int_{\R^n} |K(\ve x) - K(0)| U(x)^{p^*_s}\, dx = \left(\int_{|x|\le R} + \int_{|x|>R}\right)|K(\ve x) - K(0)| U(x)^{p^*_s}\, dx  = I + II.
$$
For $I$ we use \eqref{Est.K} and obtain
$$
I\le C (\ve R)^\alpha S^\frac{n}{sp}.
$$
For $II$, we just use the $L^\infty$ bound of $K$ and \eqref{c12} to obtain
$$
II\le C\int_{|x|>R} V(x)^{p^*_s}\, dx \le C R^{-\frac{n}{p-1}}.
$$
Now, optimizing on $R$, we take 
$$
R=\ve^{-\frac{\alpha(p-1)}{\alpha(p-1) + n}},
$$
obtaining
$$
\int_{\R^n} |K(\ve x) - K(0)| U(x)^{p^*_s}\, dx\le C \ve^\frac{\alpha n}{\alpha(p-1) + n}.
$$
From this estimate, we obtain the desired result once observed that  $\frac{\alpha n}{\alpha(p-1) + n}>sp$ if and only if $\alpha>\alpha^*.$
\end{proof}

As a corollary, we obtain the next result

\begin{teo}\label{aplicacion.2.5}
Let $0<s<1<p<\infty$ be such that $sp^2<n$. 
Assume moreover that $K\in L^\infty(\R^n)\cap C(\R^n)$ verifies \eqref{alpha} and reaches its maximum at some point $x_0\in\R^n$ for which the function $h$ satisfies \eqref{h2}. Assume also that $h$  satisfies \eqref{h1} with $q=p$. 

Then, \eqref{MainEqu} has a nontrivial solution for any $0<\lambda<\lambda_1(h)$, where $\lambda_1(h)$ is given by \eqref{eq.lambda1}.
\end{teo}

\subsection{Proof of Theorem \ref{aplicacion.infty}}

Take $x_\ve\in\R^n$ be such that $|x_\ve|\to\infty$ as $\ve\to 0$.

The proof follows exactly as the one of Theorem \ref{aplicacion.2} with test function given by $U_{\ve,x_\ve}$. So, in order to conclude we need to estimate
$$
\int_{\R^n} K(x) U_{\ve, x_\ve}(x)^{p^*_s}\, dx \quad \text{and} \quad \int_{\R^n} h(x) U_{\ve,x_\ve}(x)^p\, dx.
$$
Changing variables, we get
\begin{equation}\label{est.h.infty}
\int_{\R^n} h(x) U_{\ve,x_\ve}(x)^p\, dx = \ve^{sp}\int_{\R^n} h(x_\ve + \ve x) U(x)^p\, dx\ge A \frac{\ve^{sp}}{(2|x_\ve|)^\gamma} \int_{B_1(0)} U^p\, dx
\end{equation}
and
\begin{align*}
\int_{\R^n} K(x) U_{\ve, x_\ve}(x)^{p^*_s}\, dx &= \int_{\R^n} K(x_\ve + \ve x) U(x)^{p^*_s}\, dx\\
&= K(\infty) S^\frac{n}{sp} + \int_{\R^n} (K(x_\ve + \ve x) - K(\infty)) U(x)^{p^*_s}\, dx.
\end{align*}
Now we need to control the integral on the right hand side. So, as in the proof of Theorem \ref{aplicacion.2.5}, we take $R>0$ to be chosen later and write
\begin{align*}
\int_{\R^n} |K(x_\ve + \ve x) - K(\infty)| U(x)^{p^*_s}\, dx &= \left(\int_{|x|\le R} + \int_{|x|>R}\right) |K(x_\ve + \ve x) - K(\infty)| U(x)^{p^*_s}\, dx\\
&= I+ II.
\end{align*}
To bound $I$ we use \eqref{K.infty} and obtain
$$
I\le \frac{C}{(|x_\ve| - \ve R)^\beta} S^\frac{n}{sp}.
$$
To bound $II$ we use the $L^\infty$ bound of $K$ and get
$$
II\le C \int_{|x|>R} V^{p^*_s}\, dx \le C R^{-\frac{n}{p-1}}.
$$
So, if we take $R=|x_\ve|$ we obtain
$$
\int_{\R^n} |K(x_\ve + \ve x) - K(\infty)| U(x)^{p^*_s}\, dx \le \frac{C}{|x_\ve|^\beta} + \frac{C}{|x_\ve|^\frac{n}{p-1}}.
$$

Finally, if we take 
$$
|x_\ve| \gg \max\{\ve^{-\frac{sp}{\beta-\gamma}}, \ve^{-\frac{sp}{\frac{n}{p-1}-\gamma}}\}
$$
we arrive at
$$
\int_{\R^n} K(x) U_{\ve,x_\ve}(x)^{p^*_s}\, dx = K(\infty) S^\frac{n}{sp} + o\left(\frac{\ve^{sp}}{|x_\ve|^\gamma}\right). 
$$

This estimate, along with \eqref{est.h.infty} allows us to conclude the proof as in Theorem \ref{aplicacion.2}.\qed

\section*{Acknowledgements}

This paper was supported by grants UBACyT 20020130100283BA, CONICET PIP 11220150100032CO and ANPCyT PICT 2012-0153.

The authors are members of CONICET.

\bibliography{biblio}
\bibliographystyle{plain}

\end{document}